\documentclass[twoside, reqno, 10pt]{amsart}

\usepackage[left=2.5cm, right=2.5cm, top=3cm, bottom=2.5cm]{geometry}

\usepackage[colorlinks=true, pdfstartview=FitV, linkcolor=blue,citecolor=blue, urlcolor=blue]{hyperref}

\usepackage[usenames]{xcolor} 

\definecolor{labelkey}{rgb}{0,0,1}
\definecolor{Red}{rgb}{0.7,0,0.1}
\definecolor{Green}{rgb}{0,0.7,0}

\usepackage{amsfonts, amssymb, amsmath, amsthm, mathrsfs, bbm, cjhebrew, gensymb, textcomp, mathtools, dsfont, calligra}

\usepackage[normalem]{ulem}

\usepackage{commath, setspace, subcaption, parcolumns, multirow, multicol, accents, comment, marginnote, verbatim, empheq, enumerate, stackrel, enumitem, float, physics}
\usepackage[capitalize,nameinlink,noabbrev]{cleveref}

\usepackage{graphicx, graphics, epsfig, psfrag, tikz, tikz-cd,svg}

\numberwithin{equation}{section}

\usepackage{tikz}

\newtheorem{Thm}{Theorem}[section]

\newtheorem{Prop}[Thm]{Proposition}
\newtheorem{Cor}[Thm]{Corollary}

\newtheorem{Rmk}[Thm]{Remark}

\newtheorem*{Thm*}{Theorem}

\newcommand{\Z}{\mathbb{Z}}

\DeclareMathOperator{\esssup}{ess\ sup}

\newcommand{\til}[1]{{\tilde{#1}}}
\newcommand{\Sob}[2]{\lVert#1\rVert_{#2}}

\newcommand{\goesto}{\rightarrow}
\newcommand{\smod}{\setminus}

\newcommand{\al}{\alpha}
\newcommand{\be}{\beta}

\newcommand{\De}{\Delta}
\newcommand{\gam}{\gamma}

\newcommand{\s}{\sigma}

\newcommand{\kap}{\kappa}

\newcommand{\Om}{\Omega}

\newcommand{\bdy}{\partial}
\newcommand{\lb}{\langle}
\newcommand{\rb}{\rangle}

\newcommand{\ts}{\tilde{\sigma}}

\newcommand{\tG}{\tilde{G}}

\usepackage{todonotes}

\usepackage{multicol}
\usepackage[normalem]{ulem}

 \title[On reconstruction of unknown external forces via low-mode observations in 2D NSE]{On the reconstruction of unknown driving forces from low-mode observations in the 2D Navier-Stokes Equations}
 \author{Vincent R. Martinez}
\begin{document}
\maketitle

\begin{abstract}
This article is concerned with the problem of determining an unknown source of non-potential, external time-dependent perturbations of an incompressible fluid from large-scale observations on the flow field. A relaxation-based approach is proposed for accomplishing this, which makes use of a nonlinear property of the equations of motions to asymptotically enslave small scales to large scales. In particular, an algorithm is introduced that systematically produces approximations of the flow field on the unobserved scales in order to generate an approximation to the unknown force; the process is then repeated to generate an improved approximation of the unobserved scales, and so on. A mathematical proof of convergence of this algorithm is established in the context of the two-dimensional Navier-Stokes equations with periodic boundary conditions under the assumption that the force belongs to the observational subspace of phase space; at each stage in the algorithm, it is shown that the model error, represented as the difference between the approximating and true force, asymptotically decreases to zero in a geometric fashion provided that sufficiently many scales are observed and certain parameters of the algorithm are appropriately tuned. 
\end{abstract}

\vspace{1em}

{\noindent \small {\it {\bf Keywords: inferring unknown external force, parameter estimation, model error, system identification, data assimilation, feedback control, nudging, synchronization, Navier-Stokes equations, convergence, sensitivity}
  }  
  }

\section{Introduction}

In the derivation of any model, parameters arise that capture intrinsic properties of the phenomenon of interest. In the case of modelling a turbulent, incompressible fluid flow via the Navier-Stokes equations (NSE), assuming a constant density, the two relevant parameters are essentially the kinematic viscosity of the fluid and the external force. For instance, in eddy diffusivity models, turbulent viscosity coefficients must be specified \cite{BerselliIliescuLaytonBook}. In practice, these parameters must be empirically determined. On the other hand, in proposing any turbulent closure, one inevitably commits model errors. These errors may themselves then be modeled as a body force, whose exact form in terms of the mean field is unknown.  

In the current work, we propose an algorithm for determining all large-scale features of an external driving force in the two-dimensional incompressible NSE (2D NSE) down to the observation scale. In practice, such an algorithm can be used to filter model errors that are represented as external driving forces. In our idealized set-up, the fluid occupies a periodic domain, $\Om=[0,2\pi]^2$, the density of the fluid is normalized to unity, the kinematic viscosity of the fluid is known perfectly, but the external driving force is not. We point out that the problem of determining the viscosity based on direct observation of the velocity field was recently studied in \cite{CialencoGlattHoltz2011, DiLeoniClarkMazzinoBiferale2018, CarlsonHudsonLarios2020, Martinez2022, BiswasHudson2023}, where estimators for the viscosity were proposed, their consistency and asymptotic normality were established \cite{CialencoGlattHoltz2011}, convergence analyses for viscosity-recovery algorithms were carried out \cite{Martinez2022}, and numerical tests were performed \cite{DiLeoniClarkMazzinoBiferale2018, CarlsonHudsonLarios2020}. The problem of multi-parameter recovery in chaotic dynamical systems was studied in \cite{CarlsonHudsonLariosMartinezNgWhitehead2021, PachevWhiteheadMcQuarrie2021}, while the recent work \cite{CarlsonLarios2021} studies the sensitivity of the 2D NSE to changes in the viscosity, as well as its implications for a certain downscaling algorithm for data assimilation.

Generally speaking, one of the objectives of this article is to study the extent to which external forces can be determined based on error-free, but partial observations of flow field through a practically implementable algorithm. In our ideal set-up, we will assume that we have access to a time series of the velocity field, but only through sufficiently small, but nevertheless finitely many, length scales. Ultimately, we introduce an algorithm for reconstructing large-scale features of the unknown force and establish its convergence under the assumption that the force acts only on length scales that are directly observed, that is to say, that the force belongs to the span of the observational field.

To be more precise, we recall that the 2D NSE on $\Om=[0,2\pi]^2$ is given by
    \begin{align}\label{eq:nse:intro}
        \bdy_tu+(u\cdotp\nabla) u=\nu\De u-\nabla p+f,\quad \nabla\cdotp u=0,
    \end{align}
where the kinematic viscosity, $\nu>0$, is given and fixed, and the time-series, $\{P_Nu(t)\}_{t\in[0,T]}$, is known, up to some time $T>0$, where $P_N$ is the $L^2$-orthogonal projection onto Fourier wave-numbers corresponding to $|k|\leq N$, and $u\cdotp\nabla=u^j\bdy_j$, where repeated indices indicates summation over those indices. However, the external force, $f$, is time-independent, but unknown. The scalar pressure field is denoted by $p$; upon taking the divergence of \eqref{eq:nse:intro}, one sees that $p$ is determined entirely by $u, f$ via the Poisson equation $-\De p=\bdy_i\bdy_j(u^iu^j)+\bdy_jf^j$. We will assume that $u$, $p$, and $f$ are both mean-free and periodic over $\Om$, and that $f$ is divergence-free. The main result is that for forces satisfying $P_Nf=f$, where $N$ is sufficiently large, there exists a sequence of times $\{t_n\}_n$ and sequence of approximating forces,  $\{f_n\}_n$, depending only on the observations $\{P_Nu(t)\}_{t\geq0}$ such that $f_n$ converges to $f$. We will in fact address the more general case of time-dependent forces.

The motivating idea for the algorithm we propose is based on the notion of ``asymptotic completeness" for nonlinear dissipative systems. For systems that possess this property, it asserts that having direct (observational) access to a sufficiently rich, but nevertheless finite-dimensional, set of scales, is enough to asymptotically determine the unobserved scales. This property was rigorously shown to hold for \eqref{eq:nse:intro} by Foias and Prodi \cite{FoiasProdi1967} in the case when access to sufficiently many Fourier modes of the velocity field is available; for the case of three-dimensions, the reader is referred to the works \cite{CheskidovDai2017, BiswasPrice2021}. Specifically, given two solutions $u_1, u_2$ of \eqref{eq:nse:intro} corresponding to external forces $f_1, f_2$, it was shown that there exists $N\geq1$, depending on $\nu$ and $f_1,f_2$, but only through their size, such that $P_N(u_1(t)-u_2(t))\goesto0$ and $f_1(t)-f_2(t)\goesto0$ as $t\goesto\infty$ together imply $u_1(t)-u_2(t)\goesto0$ as $t\goesto\infty$. In this case, it is then said that the Fourier modes corresponding to wavenumbers $|k|\leq N$ are \textit{determining} for the system \eqref{eq:nse:intro}; the smallest such number, $N$, is referred to as the number of \textit{determining modes}. In light of this result, one sees that the problem of inferring an unknown force may be possible to solve provided that sufficiently many modes are observed and that one identifies an algorithm that asymptotically reconstructs higher-modes that can subsequently be made use of to approximate the force. 

To see the main difficulty that must be overcome in doing so, suppose that one is given access to $P_Nu(t)$, for $t\geq t_0$. Then a na\"ive, but reasonable first approximation to the low-modes of the force may be given by simply evaluating the nonlinear differential operator determined by \eqref{eq:nse:intro} along $P_Nu$. This yields
    \begin{align}\label{def:force:initial}
        \bdy_tP_Nu-\nu\De P_Nu+P_N((P_Nu\cdotp\nabla)P_Nu)+\nabla p_N=:f_0,
    \end{align} 
where $p_N$ is found by enforcing $f_0$ to be divergence-free. On the other hand, by applying the low-pass filter $P_N$ to \eqref{eq:nse:intro}, one also obtains
    \begin{align}\label{def:nse:low:pass}
        \bdy_tP_Nu-\nu\De P_Nu+P_N((P_Nu\cdotp\nabla)P_Nu)+\nabla p_N=P_Nf+\mathcal{R}_N,
    \end{align}
where $\mathcal{R}_N$ denotes the Reynolds stress, defined as
    \begin{align}\label{def:stress}
        \mathcal{R}_N:=P_N((P_Nu\cdotp\nabla)P_Nu)-P_N((u\cdotp\nabla)u)+\nabla p_N-\nabla P_Np.
    \end{align}
One therefore has the following relation    
    \begin{align}\label{eq:rel:reynolds}
        \mathcal{R}_N=f_0-P_Nf.
    \end{align}
In particular, determination of $P_Nf$ is equivalent to determination of $\mathcal{R}_N$. The fundamental difficulty of the problem arises from the elementary fact that $\mathcal{R}_N$ depends on both $P_Nu$ and $(I-P_N)u$, which is tantamount to the closure problem of turbulence. In comparison, if \eqref{eq:nse:intro} were replaced by the linear heat equation, then observation of $P_Nu$ along a trajectory of $u$ provides an exact reconstruction for $P_Nf$ along that same trajectory. Thus, any resolution of this difficulty must address a way to reconstruct high-mode information from low-mode data. In \cref{sect:algorithm}, we introduce an algorithm that, under suitable conditions, accomplishes this and systematically decrements the ``Reynolds stress'' at each stage in a geometric fashion, yielding a convergent scheme for approximating the force on the low-modes. From this point of view, the proposed scheme may be viewed as a ``nonlinear filtering'' of large-scale error.

For systems that possess an \textit{inertial manifold},  reconstruction of small scales from large scales would be possible. Indeed, the existence of an inertial manifold implies a strong form of enslavement of scales as it asserts the existence of a map $\Phi$ such that $P_Nu|_{t=t_0}\mapsto (I-P_N)u|_{t\geq t_0}$, that is, $u(t)=P_Nu(t)+\Phi(P_Nu(t_0))(t)$, for all $t\geq t_0$. In other words, knowledge of low-modes at a single time is sufficient to determine the high-mode behavior at all future times. The inertial manifold would then be determined by the graph of $\Phi$ (see \cite{FoiasTemam1985, FoiasSellTemam1988}). However, the existence of such a map for \eqref{eq:nse:intro} remains an outstanding open problem. One of the main points of this article is that one need only rely on the weaker property of asymptotic completeness to reconstruct sufficiently high-mode components of the state variables in order to eventually recover the low-modes of the forcing.

There are at least two ways available in the literature for doing, one due to Foias, Jolly, Kravchenko, and Titi in \cite{FoiasJollyKravchenkoTiti2012} and another by Azouani, Olson, and Titi in \cite{AzouaniOlsonTiti2014}. The first construction \cite{FoiasJollyKravchenkoTiti2012} allows one to encode projections of solutions on the global attractor of \eqref{eq:nse:intro} as traveling wave solutions to some infinite-dimensional ordinary differential equation. However, it effectively requires one to solve the evolution equation of \eqref{eq:nse:intro} corresponding to the high-modes \cite{FoiasJollyLithioTiti2017}. The second method introduces downscaling algorithm for data assimilation \cite{AzouaniOlsonTiti2014}, in which large-scale observations are exogenously inserted into \eqref{eq:nse:intro} as a feedback-control term that serves to drive the corresponding solution of this modified system towards the reference solution, but only on large scales. By virtue of the Foias-Prodi property of determining modes, it was then shown that by tuning the strength of the feedback-control system appropriately, the generated approximating signal asymptotically synchronizes to the reference signal to which the large-scale observations correspond. It is the latter approach that we will make use of to systematically reconstruct high-mode information on the reference field. 

The remainder of the manuscript is organized as follows. In \cref{sect:notation}, we establish the notation and functional setting in which the result will be proved, as well as classical well-posednesss results that we will make use of. A derivation of the algorithm that reconstructs the external force is presented in \cref{sect:algorithm}. We provide formal statements of the convergence results in \cref{sect:results}, then outline their proof in \cref{sect:outline}. The main step is to reduce the analysis of the convergence to establishing suitable ``sensitivity estimates". These estimates are proved in \cref{sect:sensitivity}. We finally supply rigorous proofs of the main results in \cref{sect:proof}. Technical auxiliary results are relegated to \cref{sect:app:apriori}.

\section{Notation and Functional Setting}\label{sect:notation}
Let $\Om=[0,2\pi]^2$ and $L^p_\s=L^p_\s(\Om)$, for $p\in[1,\infty]$, denote the space of $p$-integrable (in the sense of Lebesgue), mean-free, solenoidal vector fields over $\Om$, which are periodic in each direction; its norm is given by
    \begin{align}\label{def:Lp:norm}
        \Sob{u}{L^p}^p:=\int_{\Om} |u(x)|^pdx,\quad \text{for}\ p\in[1,\infty),\quad
        \Sob{u}{L^\infty}:=\esssup_{x\in\Om}|u(x)|,
    \end{align}
where $\esssup$ denotes the essential supremum. We let $H^k_\s=H^k_\s(\Om)$ denote the space of periodic, mean-free, solenoidal vector fields over $\Om$ whose weak derivatives (in the sense of Sobolev) up to order $k$ belong to $L^2_\s$; its norm is given by
    \begin{align}\label{def:Hk:norm}
        \Sob{u}{H^k}^2:=\sum_{|\al|\leq k}\int_{\Om}|\bdy^\al u(x)|^2dx,
    \end{align}
where $\al\in(\mathbb{N}\cup\{0\})^2$ is a multi-index. We will abuse notation and use $L^2_\s, H^k_\s$ to denote the corresponding spaces for scalar functions as well; in this case, $|u|$ is interpreted as absolute value, rather than Euclidean norm.

We will make use of the functional form of the Navier-Stokes equations, which is given by
    \begin{align}\label{eq:nse}
        \frac{d}{dt}u+\nu Au+B(u,u)=Pf,
    \end{align}
where 
    \[
        Au=-P\De u\quad\text{and}\quad B(u,v)=P(u\cdotp\nabla) v,
    \]
and $P$ denotes the Leray-Helmholtz projection, that is, the orthogonal projection onto divergence-free vector fields; we refer to $A$ as the Stokes operator. Note that by orthogonality, $\Sob{P}{L^2\goesto L^2}\leq 1$. Moreover
    \begin{align}\label{def:LP}
        \widehat{Pu}(n):=\hat{u}(n)-\frac{n\cdotp\hat{u}(n)}{|n|^2}n,\quad n\in\Z^2\smod\{(0,0)\},
    \end{align}
where $\hat{u}(n)$ denotes the Fourier coefficient of $u$ corresponding to wavenumber $n\in\Z^2$. Also, powers of $A$ can be defined spectrally via
    \begin{align}\label{def:A}
        \widehat{A^{k/2}u}(n):=|n|^k\hat{u}(n).
    \end{align}
Hence, $P$ commutes with $\De$ in the setting of periodic boundary conditions. Recall that for $u\in H^1_\s$ the Poincar\'e inequality states
    \begin{align}\label{eq:poincare}
        \Sob{u}{L^2}\leq \Sob{A^{1/2} u}{L^2}=\Sob{A^{1/2}u}{L^2}.
    \end{align}
It follows that for each integer $k\geq1$, there exists a constant $c_k>0$ such that
    \begin{align}\label{est:equiv:Sob}
        c_k^{-2}\Sob{u}{H^k}^2\leq \sum_{|\al|=k}\int_\Om|\bdy^\al u(x)|^2dx\leq c_k^2\Sob{u}{H^k}^2,
    \end{align}
whenever $u\in H^k_\s$. In particular, by the Parseval identity, it follows that there also exist constants $c_k>0$ such that
     \begin{align}\label{est:equiv:A}
        c_k^{-1}\Sob{u}{H^k}\leq \Sob{A^{k/2}u}{L^2}\leq c_k\Sob{u}{H^k},
    \end{align}
for all $u\in H^k_\s$.

Given $t_0\geq0$ and $f\in L^\infty(t_0,\infty;L^2_\s)$, we define the Grashof-type number by
    \begin{align}\label{def:Grashof}
        \tG:=\frac{\sup_{t\geq t_0}\Sob{f(t)}{L^2}}{\kap_0^2\nu^2},
    \end{align}
where $\kap_0=(2\pi)/L$, where $L$ is the linear length of the domain; since $\Om=[0,2\pi]^2$, we see that $\kap_0=1$. We note that the Grashof number is traditionally denoted by $G$, that is, undecorated, when the force is \textit{independent of time}. Since we allow for time-dependence in the force, we will distinguish between the two notations by making use of tilde. See, for instance \cite{Balci_Foias_Jolly_2010}, where this distinction is also maintained.

Observe that $Pf=f$, for any $f\in L^2_\s$. One has the following classical results regarding the existence theory for \eqref{eq:nse} (\cite{ConstantinFoias1988, TemamBook1997, TemamBook2001, FoiasManleyRosaTemamBook2001}).

\begin{Thm}\label{thm:nse:wellposed}
Let $t_0\geq 0$ and $f\in L^\infty(t_0,\infty;L^2_\s)$. For all $u_0\in H^1_\s$, there exists a unique $u\in C([t_0,T];H_\s^1)\cap L^2(t_0,T;H^{2})$ satisfying \eqref{eq:nse}, for all $T>0$, such that $\frac{d}{dt}u\in L^2(t_0,T;L^2_\s)$ and
    \begin{align}\label{est:nse:H1}
        \Sob{A^{1/2}u(t)}{L^2}^2\leq \Sob{A^{1/2} u(t_0)}{L^2}^2e^{-\nu (t-t_0)}+\nu^2\tilde{G}^2(1-e^{-\nu(t-t_0)}),
    \end{align}
for all $t\geq t_0\geq0$.
\end{Thm}

From \eqref{est:nse:H1}, we see that for $t_0=0$ and all $t>0$ sufficiently large, depending on $\Sob{A^{1/2}u_0}{L^2}$, one has
    \begin{align}\label{def:rad:H1}
        \Sob{A^{1/2}u(t)}{L^2}\leq \sqrt{2}\nu \tG=:\til{c}_1\til{R}_1,\quad \til{c}_1=\sqrt{2}.
    \end{align}
Let $\til{B}_1$ denote the ball of radius $\tilde{R}_1$, centered at the origin in $L^2_\s$. Observe that $A^{1/2}u_0\in \tilde{B}_1$ implies  via \eqref{est:nse:H1} that $u(t;u_0,f)\in \tilde{B}_1$, for all $t\geq 0$, where $u(t;u_0,f)$ denotes the unique solution of \eqref{eq:nse} with initial data $u_0$ and external forcing $f$. Moreover, there exists a constant $\til{c}_2>0$ such that if $f\in L^\infty(0,\infty;H^1_\s)$ and  $Au_0\in\al\til{B}_2$, where $\al>0$ is arbitrary, $\til{B}_2$ denotes the ball of radius $\tilde{R}_2=\til{c}_2\nu(\tG+\ts_1)\tG$, centered at the origin in $L^2_\s$, and $\al\til{B}_2$ the same ball of radius $\al\til{R}_2$, then
    \begin{align}\label{def:rad:H2:gen}
    \Sob{Au(t)}{L^2}^2\leq (1+\al^2)\tilde{c}_2^2\nu^2(\ts_1+\tG)^2\tG^2,
    \end{align}
for all $t\geq0$, where $\ts_1$ denotes a ``shape factor" defined by
     \begin{align}\label{def:shape}
        \ts_1:=\frac{\sup_{t\geq t_0}\Sob{A^{1/2} f(t)}{L^2}}{\sup_{t\geq t_0}\Sob{f(t)}{L^2}}.
    \end{align}
In other words, $Au_0\in \al\til{B}_2$ implies $Au(t)\in(1+\al^2)^{1/2}\til{B}_2$, for all $t\geq0$. Observe that $\ts_1\geq1$ by Poincar\'e's inequality. In particular, if $Au_0\in\til{B}_2$, then 
    \begin{align}\label{est:rad:H2:gen}
        \Sob{Au(t)}{L^2}\leq \sqrt{2}\til{c}_2\nu(\tG+\ts_1)\tG=\sqrt{2}\til{R}_2.
    \end{align}
Bounds sharper than \eqref{est:rad:H2:gen} were established in \cite{DascaliucFoiasJolly2005} in the setting where $f$ was time-independent. In this particular case, one has that
    \begin{align}\label{def:rad:H2}
        \Sob{Au(t)}{L^2}\leq c_2\nu(\s_{1}^{1/2}+G)G=:R_2,
    \end{align}
holds for all $t\geq0$, for some $c_2>0$, provided that $u_0\in B_2$, the ball of radius $R_2$ in $H^2_\s$. Here, $G$ denotes the Grashof number, which is simply given by \eqref{def:Grashof} when $f$ is time-independent. Similarly, $\s_1$ is given by \eqref{def:shape} when $f$ is time-independent. For the sake of completeness, we supply a short proof of \eqref{def:rad:H2} in \cref{sect:app:apriori}.

\section{Description of the Algorithm}\label{sect:algorithm}

We consider the following feedback control system 
    \begin{align}\label{eq:nse:ng}
        \frac{d}{dt}v+\nu Av+B(v,v)=h-\mu P_N(v-u),
    \end{align}
where $h$, possibly time-dependent, is given. The well-posedness theory and synchronization properties of this model was originally developed in \cite{AzouaniOlsonTiti2014} for a more general class of observables, which includes projection onto finitely many Fourier modes as a special case, in \cite{BlomkerLawStuartZygalakis2013, BessaihOlsonTiti2015} in the setting of noisy observations, while the issue of synchronization in higher-order topologies was studied in \cite{BiswasMartinez2017, BiswasBrownMartinez2022}. In the idealized setting considered in this article, the existence, uniqueness results in \cite{AzouaniOlsonTiti2014} will suffice for our purposes. This is stated in the following theorem.

\begin{Thm}\label{thm:nse:ng:wellposed}
Let $t_0\geq0$ and $h\in L^\infty(t_0,\infty;L^2_\s)$. Let $u$ denote the unique solution of \eqref{eq:nse} corresponding to initial data $u_0\in H^1_\s$ guaranteed by \cref{thm:nse:wellposed}. There exists a constant $\til{c}>0$ such if $\mu,N$ satisfy
    \begin{align}\label{cond:mu:N:ng}
        \mu \leq \til{c}\nu N^2
    \end{align}
then given $v_0\in H^1_\s$, there exists a unique solution, $v$, to the initial value problem corresponding to \eqref{eq:nse:ng} such that
    \begin{align}
        v\in C([t_0,T];H^1_\s)\cap L^2(t_0,T; H^2_\s)\quad\text{and}\quad \frac{d}{dt}v\in L^2(t_0,T;L^2_\s),
    \end{align}
for any $T>0$.
\end{Thm}

The feedback control system \eqref{eq:nse:ng} was originally conceived as a way to assimilate observations on the flow field into the equations of motion in order to reconstruct the unobserved scales of motion. There is a considerable body of work studying the extent to which this is possible in various situations in hydrodynamics such as Rayleigh-B\'enard convection \cite{FarhatJollyTiti2015, FarhatLunasinTiti2016b, AltafTitiGebraelKnioZhaoMcCabe2017, FarhatLunasinTiti2017a, FarhatJohnstonJollyTiti2018, FarhatGlattHoltzMartinezMcQuarrieWhitehead2020, CaoJollyTitiWhitehead2021}, turbulence \cite{AlbanezNussenzveigLopesTiti2016, FarhatLunasinTiti2017b, LariosPei2020, ClarkDiLeoniMazzinoBiferale2020, ChenLiLunasin2021, CaoGiorginiJollyPakzad2022, ZaunerMonsMarquetLeclaire2022}, geophysical fluids \cite{FarhatLunasinTiti2016c, JollyMartinezTiti2017, AlbanezBenvenutti2018, JollyMartinezOlsonTiti2019, Pei2019, DesamsettiDasariLangodanTitiKnioHoteit2019}, dispersive equations \cite{JollySadigovTiti2015, JollySadigovTiti2017}, as well as various numerical analytical and computational studies \cite{GeshoOlsonTiti2016, FoiasMondainiTiti2016, BlocherMartinezOlson2018, IbdahMondainiTiti2019, LariosRebholzZerfas2019, CelikOlsonTiti2019, Garcia-ArchillaNovo2020, Garcia-ArchillaNovoTiti2020, DiegelRebholz2022}.

Given $h$, we may thus obtain from \eqref{eq:nse:ng} an approximate reconstruction of the high-modes of $u$ via $Q_Nv=(I-P_N)v$. We therefore propose the following algorithm: Let $h=f_0$ denote the initial guess for the forcing field, defined for all $t\geq t_0$, for some fixed initial time $t_0\geq0$, and let $v^0$ to be an arbitrary initial state; $f_0$ is considered to be the approximation at stage $n=0$ that is prescribed by the user and may, in fact, be chosen arbitrarily. Suppose that $P_Nf_0=f_0$. Then at stage $n=1$, we consider
    \begin{align}\label{eq:step:1}
        \frac{d}{dt}v_1+\nu Av_1+B(v_1,v_1)=f_0-\mu P_N(v_1-u),\quad v_1(t_0)=v_1^0,\quad t\in I_0:=[t_0,\infty).
    \end{align}
We define the first approximation to the flow field by
    \begin{align}\label{def:vel:1}
    u_1=P_Nu+Q_Nv_1, \quad \text{for}\  t\in I_0.
    \end{align}
By design, $u_1$ will relax towards $u$ after a transient period, $\rho_1:=t_1-t_0$, that is proportional to the relaxation time-scale, $\mu^{-1}$, where $t_1\gg t_0$, but only up to an error of size $O(g_0)$, where $g_0:=f_0-f$, which represents the ``model error.'' Only after this period has transpired will we extract the first approximation to $f$ via the formula
    \begin{align}\label{eq:update:1}
        f_1(t):=\frac{d}{dt}P_Nu_1+\nu AP_Nu_1+P_NB(u_1,u_1),\quad \text{for all}\  t\in I_1:=[t_0+\rho_1,\infty),\ \text{for some}\ \rho_1\gg 0.
    \end{align}
Observe that $P_Nf_1=f_1$. To obtain new approximations in subsequent stages, we proceed recursively: Suppose that at stage $n-1$, a force, $f_{n-1}=P_Nf_{n-1}$ over $t\in I_{n-1}:=[t_{n-2}+\rho_{n-1},\infty)$, has been produced, where the relaxation period satisfies $\rho_{n-1}\gg 0$, and an arbitrary initial state, $v_n^0$, has been given. Then consider
    \begin{align}\label{eq:step:n}
        \frac{d}{dt}v_n+\nu Av_n+B(v_n,v_n)=f_{n-1}-\mu P_N(v_n-u),\quad v_n(t_{n-1})=v_n^0,\quad t\in I_{n-1}.
    \end{align}
Let $u_n=P_Nu+Q_Nv_n$, for $t\in I_{n-1}$, and define the approximation to the force at stage $n$ by
    \begin{align}\label{eq:update:n}
        f_n(t):=\frac{d}{dt}P_Nu_n+\nu AP_Nu_n+P_NB(u_n,u_n),\quad \text{for all}\ t\in I_n:=[t_{n-1}+\rho_n,\infty),\ \text{for some}\ \rho_n\gg 0.
    \end{align}
This procedure produces a sequence of forces $f_1|_{I_1}, f_2|_{I_2}, f_3|_{I_3},\dots$ that approximates $f$ on time intervals $I_n$ whose left-hand endpoints are increasing with $n$. In particular, the sequence $\{f_n|_{I_n}\}_n$ \textit{asymptotically} approximates $f$.

The key step to ensuring convergence of the generated sequence $\{f_n\}_{n\geq1}$ to the true forcing, $f$, is to control model errors, $g_n:=f_n-f$, at each stage, in terms of the synchronization errors, $w_n$, which, in turn, are controlled by the model error \textit{from the previous stage}; this will be guaranteed to be the case after transient periods of length $\rho_n:=t_n-t_{n-1}$, which allows relaxation in \eqref{eq:step:1} to take place. However, it will be shown that this convergence can only be guaranteed to occur on the ``observational subspace'' $P_NL^2_\s$, for $N$ sufficiently large. Indeed, a crucial observation at this point is that if $f=P_Nf$, then $f_n=P_Nf_n$, for all $N\geq1$. We refer the reader to \cref{rmk:marchioro} and \cref{rmk:high} for additional remarks on the basic expectations for recovering force from low-mode data and the underlying limitations of this algorithm. We refer the reader to \cref{rmk:relaxation} for a more detailed discussion on the size of the transient periods, $\rho_n$.

\begin{Rmk}\label{rmk:f0}
We remark that the first guess, $f_0$, need not be arbitrary and can be chosen to be
    \begin{align}\label{def:force:initial:rmk}
        f_0:=\frac{d}{dt}P_Nu+\nu AP_Nu+P_NB
        (P_Nu,P_Nu),
    \end{align}
as suggested in \eqref{def:force:initial}, where $t_{0}>0$ is chosen sufficiently large so that $u(t_0)$ is contained in an absorbing ball for \eqref{eq:nse}.  Similarly, the initial states, $v_n^0$, at each stage need not be arbitrary. Since $\{P_Nu(t)\}_{t\geq t_0}$ is assumed to be given, one can initialize the system governing $v_n$ at time $t=t_{n-1}$ with $v_n^0=P_Nu(t_{n-1})$ at each stage. These natural choices would presumably aid in the implementation of the proposed algorithm; we refer the reader \cref{rmk:relaxation} for additional remarks related to this point.
\end{Rmk}

Before outlining the proofs, we rigorously state the main results of the article.

\section{Statements of Main Results}\label{sect:results}
Let $A^{1/2}u_0\in \tilde{B}_1$ and $Au_0\in\tilde{B}_2$ and $f\in   C([t_0,\infty);L^2_\s)\cap L^\infty(t_0,\infty;L^2_\s)$, for some $t_0\geq0$, where $\tilde{B}_1 ,\tilde{B}_2$ were defined as in \cref{sect:notation}. Let $u$ denote the unique strong solution of \eqref{eq:nse} corresponding to forcing $f$ and initial data $u_0$ guaranteed by \cref{thm:nse:wellposed} corresponding to initial velocity $u_0$ and external forcing $f$. Given $t_0\geq 0$, let $\gam_0$ denote the initial relative error defined by
    \begin{align}\label{def:error:rel}
    \gam_0:=\frac{\sup_{t\geq t_0}\Sob{f_0(t)-f(t)}{L^2}}{\sup_{t\geq t_0}\Sob{f(t)}{L^2}},
    \end{align}
where $f_0$ is a user-prescribed initial guess for the force which satisfies $f_0\in C([t_0,\infty);L^2_\s)\cap L^\infty(t_0,\infty;L^2_\s)$. Recall that the assumptions on the initial data imply (via \eqref{est:nse:H1}, \eqref{def:rad:H2}) that $A^{1/2}u(t)\in\tilde{B}_1$ and $Au(t)\in2\tilde{B}_2$, for all $t\geq0$.

\begin{Thm}\label{thm:low}
Suppose that $f=P_Nf$, for some $N\geq1$, and that $\{P_Nu(t)\}_{t\geq t_0}$ is given. There exists a positive constant $c$, depending on $\gam_0$, such that for any $\be\in(0,1)$, if $N\geq 1$ satisfies
    \begin{align}\label{cond:mu:N:low}
         \be^{2}N^2>c(\ts_1+\tG)\tG^2,
    \end{align}
then there exists a choice for the tuning parameter $\mu$ and increasing sequence of times  $t_{n-1}\leq t_n$, for $n\geq2$, such that
    \begin{align}\label{est:error:recursion}
        \sup_{t\geq t_n}\Sob{f_n(t)-f(t)}{L^2}\leq \be\left(\sup_{t\geq t_{n-1}}\Sob{f_{n-1}(t)-f(t)}{L^2}\right),
    \end{align}
for all $n\geq1$, where $f_n=P_Nf_n$, $f_n\in C([t_0,\infty);L^2_\s)\cap L^\infty(t_0,\infty;L^2_\s)$, and each $f_n$ is given by \eqref{eq:update:n}.
\end{Thm}

\begin{Rmk}\label{rmk:condition}
It is worth pointing out that \eqref{cond:mu:N:low} depends on the the unknown forcing $f$. However, it must be emphasized that this condition only depends on $f$ through its size. From a practical perspective, one must always approach the problem of parameter estimation with prior knowledge in hand. In this light, what condition \eqref{cond:mu:N:low} indicates is that if one has access to the Grashof number of the flow, for instance through measurement of the Reynolds number (see \cref{rmk:future}), then the only prior knowledge on the force that is needed to achieve exact recovery is knowledge of its shape factor $\til{\s}_1$. It moreover indicates that having such knowledge in one's possession can quantitatively inform what balance is needed between the number of observations and the algorithmic parameter, $\mu$, to ensure a full reconstruction of the unknown force.
\end{Rmk}

\begin{Rmk}\label{rmk:cond}
 Regarding condition \eqref{cond:mu:N:low}, we remind the reader that $\tG$ also depends on viscosity. Thus, the intuition that the number of observations needed should increase as the viscosity decreases or that fewer observations are needed when viscosity is large is reflected in the statement.
\end{Rmk}

Note that when $f$ is time-dependent, \cref{thm:low} only asserts recovery of the external force asymptotically in time. However, when the force is time-periodic or time-independent, then \cref{thm:low} immediately implies that the external force is eventually recovered; we provide a statement of the time-independent case in the following corollary since the corresponding approximating forces are obtained by evaluating the sequence of approximating forces asserted in \cref{thm:low} at certain times.

\begin{Cor}\label{cor:low:time:independent}
Suppose that $f=P_Nf$, for some $N\geq1$, is time-independent, and that $\{P_Nu(t)\}_{t\geq t_0}$ is given. There exists a positive constant $c_0$, depending on $\gam_0$, such that for any $\be\in(0,1)$, if $N\geq 1$ satisfies
    \begin{align}\label{cond:mu:N:low:time:independent}
        \be^2N^2>c_0(\ts_1+G)G^2,
    \end{align}
then there exists a choice for the tuning parameter $\mu$  and an increasing sequence of times $t_{n-1}\leq t_n$ such that
    \begin{align}\label{est:error:recursion:time:independent}
        \Sob{f_n-f}{L^2}\leq \be\Sob{f_{n-1}-f}{L^2},
    \end{align}
for all $n\geq1$, where $f_n=P_Nf_n$ and each $f_n$ is given by \eqref{eq:update:n} evaluated at $t=t_n$.
\end{Cor}

In the setting of time-independent forcing, one can in fact ``recycle'' the data provided that a sufficiently long time-series is available. By ``recycle" we mean that once an approximation to the force is proposed by the algorithm at a given stage, we may use this proposed forcing to re-run the algorithm over \textit{the same time window} to produce the subsequent approximation of the force in the following stage, and so on. In this way, the same data set is used over and over in order to generate new approximations to the force.

\begin{Thm}\label{thm:recycle}
Let $T>0$. Suppose that $f=P_Nf$, for some $N\geq1$, is time-independent, and that $\{P_Nu(t)\}_{t\in[0,T)}$ is given, for some $t_0\geq0$. There exists a positive constant $c_0$, depending on $\gam_0$, such that for any $\be\in(0,1)$, if $N\geq 1$ satisfies
    \begin{align}\label{cond:mu:N:low:recycle}
        \be^2N^2>c_0(\ts_1+G)G^2,
    \end{align}
then for each $n\geq1$, there exists a choice for the tuning parameter $\mu$ and a sequence of times, $t_n>0$, such that if $T>t_0+t_n$, for all $n\geq1$, then
    \begin{align}\label{est:error:recursion:recycle}
        \Sob{f_n-f}{L^2}\leq \be\Sob{f_{n-1}-f}{L^2},
    \end{align}
for all $n\geq1$, where $f_n=P_Nf_n$, $f_n\in L^2_\s$, and each is determined by a procedure similar to that in \cref{sect:algorithm}, except that $v_n, u_n, f_n$ are always derived on the interval $[0,T)$.
\end{Thm}

The proof of \cref{thm:recycle} follows along the same lines as that of \cref{thm:low}, except that one requires a few technical modifications of the setup described in \cref{sect:algorithm}; we supply the relevant details of these modifications in \cref{sect:proof}.

\begin{Rmk}\label{rmk:marchioro} 
A well-known example given by Marchioro \cite{Marchioro1986} exhibits a scenario where observation of low-modes below the spectrum of the forcing are sufficient to determine the forcing. In particular, the work \cite{Marchioro1986} identifies a class of low-mode body forces for which the asymptotic behavior of solutions to \eqref{eq:nse} are characterized by a one-point global attractor for \eqref{eq:nse} whose unique stationary point is supported on a frequency shell strictly smaller than that of the force. Of course, this phenomenon is made possible due to presence of nonlinearity in \eqref{eq:nse}.

The extent to which it is possible to determine features of the forcing beyond the scales which are observed is, in general, not known and in particular, not addressed by \cref{thm:recycle}. Indeed, condition \eqref{cond:mu:N:low:recycle} identifies an upper bound on the number of modes that one should observe in order to uniquely determine the forcing provided that scales beyond those that are directly observed are not forced to begin with. In particular, it would be interesting to 1) study the sharpness of condition \eqref{cond:mu:N:low:recycle}, and 2) whether one may identify classes of forces, beyond the Marchioro class, which inject energy into scales larger than those which are directly observed, preferably much larger, but can nevertheless be reconstructed by these observations. These issues are left to be investigated in a future work.
\end{Rmk}

\begin{Rmk}\label{rmk:future}
There are at least three natural directions that warrant further investigation. First and foremost, a systematic computational study that probes the efficacy and limitations of this method for recovering the force should be carried out, especially in the context of noisy observations. For instance, each of theorems above suggest that a large number of modes are required to achieve convergence. In the context of a turbulent flow, $G\sim \text{Re}^2$, where $\text{Re}$ denotes the Reynolds number of the flow \cite{DascaliucFoiasJolly2008, DascaliucFoiasJolly2009}, which may be intractably large in practice. On the other hand, the analysis performed here inherently takes into account ``worst-case'' scenarios that may saturate various inequalities that were used, but which may occur rarely in reality. This direction will be explored in future work.

Secondly, the choice of Fourier modes as the form of observations is chosen due to its conceptual and analytical convenience. In principle, other observations on the velocity, such as nodal values or local spatial averages can also be used. However, the choice of Fourier modes allows one to commute the ``observation operator,'' $P_N$ with derivatives, which greatly facilitates the analysis. The failure of this commutation introduces further analytical difficulties.
Indeed, the original feedback control system was introduced with a general interpolant observable operator, $I_h$, replacing $P_N$ in \eqref{eq:nse:ng}. In connection to this, the reader is referred to the classical works \cite{FoiasTemam1984, JonesTiti1992a, JonesTiti1992b, CockburnJonesTiti1995, CockburnJonesTiti1997}, where the notion of ``finite determining parameters,'' properly generalizing ``determining modes,'' was developed.

Thirdly, although the existence of an inertial manifold for the 2D NSE is an open problem, there are several systems, which do posesses inertial manifolds \cite{FoiasSellTemam1988, TemamBook1997}, such as the Kuramoto-Sivashinsky equation \cite{FoiasNicolaenkoSellTemam1988}; it would be interesting to explore what can be gained in this particular context. In a similar vein, it would also be interesting to explore the usage of Approximate Inertial Manifolds for the 2D NSE, as it is used, for instance, in post-processing Galerkin methods. Indeed, these ideas have been successfully used in the context of downscaling data assimilation algorithms in \cite{MondainiTiti2018}.
\end{Rmk}

\section{Outline of the Convergence Argument}\label{sect:outline}
To prove \cref{thm:low}, the object of interest will be the error in the forcing, which we also refer to as ``model error.'' This is denoted by
    \begin{align}\label{def:error}
        g_n:=f_n-f,
    \end{align}
where $f_n$ is generated from the scheme described in \cref{sect:algorithm}. We claim the following: there exists a sequence of times $t_n\geq t_{n-1}$, for all $n\geq1$, such that
    \begin{align}\label{eq:converge:claim}
        \sup_{t\geq t_n}\Sob{A^{1/2} g_n(t)}{L^2}\leq \be\sup_{t\geq t_{n-1}}\Sob{A^{1/2} g_{n-1}(t)}{L^2},
    \end{align}
for some $\be\in(0,1)$. The lengths of the relaxation periods, $ \rho_n:=t_n-t_{n-1}$, between the moments, $t_{n-1}, t_n$, at which we choose to reconstitute new forces, $f_{n-1}, f_n$, respectively, are prescribed to be sufficiently large in order to allow time for the system \eqref{eq:step:n} to relax. As we will see, the length of these relaxation periods, $\rho_n$, will essentially be determined by the relaxation parameter $\mu$; the reader is referred to \cref{rmk:relaxation} for a precise relation.

Let us denote the synchronization error by
    \begin{align}\notag
        w_n:=v_n-u.
    \end{align}
Observe that the evolution of $w_n$ over the time interval $[t_{n-1},\infty)$ is governed by
    \begin{align}\label{eq:wn}
        \frac{d}{dt}w_n+\nu Aw_n+B(w_n,w_n)+DB(u)w_n=g_{n-1}-\mu P_Nw_n, \quad w_n(t_{n-1})=v_n^0-u(t_{n-1}),
    \end{align}
where $DB(u)v:=B(u,v)+B(v,u)$. Let $B_N=P_NB$ and $DB_N=P_NDB$. Then, using \eqref{eq:nse}, the facts that $P_Nf=f$ and $u_n=P_Nu+Q_Nv_N$, and  \eqref{eq:update:n}, we obtain
    \begin{align}\label{eq:error:identity}
        g_n&=\left(\frac{d}{dt}P_Nu_n+\nu AP_Nu_n+B_N(u_n,u_n)\right)-\left(\frac{d}{dt}P_Nu+\nu AP_Nu+B_N(u,u)\right)\notag\\
        &=B_N(P_Nu+Q_Nv_n,P_Nu+Q_Nv_n)-B_N(P_Nu+Q_Nu,P_Nu+Q_Nu)\notag\\
        &=B_N(Q_Nw_n,Q_Nw_N)+DB_N(u)Q_Nw_n,
    \end{align}
which holds over $[t_{n-1},\infty)$. In analogy to \eqref{eq:rel:reynolds}, we define the ``Reynolds stress'' at stage $n$ by
    \begin{align}\label{eq:error:low:modes}
        \mathcal{R}_{N}^{(n)}:=B_N(Q_Nw_n,Q_Nw_n)+DB_N(u)Q_Nw_n=g_n.
    \end{align}
Ultimately, \eqref{eq:error:low:modes} enables us to envision a recursion in the model error at each stage through the dependence of $\mathcal{R}_N^{(n)}$ on $\mathcal{R}_N^{(n-1)}$ via the synchronization error $w_n$. 
Hence, in order to prove that $\mathcal{R}_{N}^{(n)}$ vanishes in the limit as $n\goesto\infty$, we require \textit{sensitivity-type} estimates, that is, estimates on $w_n$. The estimates will take on the following form:
    \begin{align}\label{est:sensitivity:rough}
        \begin{split}
        \Sob{A^{1/2} w_n(t)}{L^2}&\leq \nu\left(\frac{\nu}{\mu}\right)^{1/2}O\left(\frac{\sup_{t\geq t_n}\Sob{g_{n-1}(t)}{L^2}}{\nu^2}\right),
        \end{split}
    \end{align}
for all $t\geq t_n$, for some sufficiently large  $t_n\geq t_{n-1}$. Before we go on to develop estimates of the form \eqref{est:sensitivity:rough} in \cref{sect:sensitivity}, let us first determine the precise manner in which their application will arise. 

\begin{Rmk}\label{rmk:convention}
In what follows and for the remainder of the manuscript, we make use of the convention that $C$ denotes a generic dimensionless constant, which may change line-to-line and possibly be large, but will always be independent of $N,\nu, \gam_0$. 
\end{Rmk}

To estimate \eqref{eq:error:low:modes}, we will invoke the following inequalities for $B(u,v)$, which follows from a direct application of H\"older's inequality:
    \begin{align}\label{est:bilinear}
        \Sob{B(u,v)}{L^2}\leq \min\{\Sob{u}{L^\infty}\Sob{A^{1/2}v}{L^2},\Sob{u}{L^4}\Sob{A^{1/2}v}{L^4}\}
    \end{align}
We will also make use of the fact that $P_N, Q_N$ commute with $A^{m/2}$, for all integers $m$, and that the following inequalities hold for all $N>0$, $\ell>0$, and $m\in\Z$:
    \begin{align}\label{est:poincare}
        \begin{split}
        \Sob{A^{m/2}P_Nv}{L^2}\leq N^m\Sob{P_Nv}{L^2},\quad \Sob{A^{m/2}Q_Nv}{L^2}\leq N^{-\ell}\Sob{A^{(m+\ell)/2}Q_Nv}{L^2}.
        \end{split}
    \end{align}
Lastly, we make the following elementary, but important observation for treating the first term appearing in \eqref{eq:error:low:modes}: for vector fields $u=(u_1,u_2)$ and $v=(v_1,v_2)$, we have
    \begin{align}\notag
        B(u,v)_i=\sum_{j=1}^2P(u\cdotp\nabla v_i)=P\nabla\cdotp(uv_i),\quad i=1,2,
    \end{align}
where $B(u,v)_i$ denotes the $i$-th component of the vector field $B(u,v)$. Thus, upon applying \eqref{est:poincare} and H\"older's inequality, we obtain
    \begin{align}\label{est:Reynolds:1}
        \Sob{B_N(u,v)}{L^2}^2\leq2\sum_{i=1}^2\Sob{P_NA^{1/2}(uv_i)}{L^2}^2\leq 2N^2\sum_{i=1}^2\Sob{uv_i}{L^2}^2\leq 2N^2\Sob{u}{L^4}^2\Sob{v}{L^4}^2.
    \end{align}

From \eqref{eq:error:low:modes}, we now apply \eqref{est:Reynolds:1}, \eqref{est:bilinear}, \eqref{est:poincare}, and interpolation to obtain
    \begin{align}\label{est:error:low:modes}
        \Sob{\mathcal{R}_{N}^{(n)}}{L^2}&\leq \Sob{P_NA^{1/2}(Q_Nw_n\otimes Q_Nw_n)}{L^2}+\Sob{DB_N(u)Q_Nw_n}{L^2}\notag\\
        &\leq CN\Sob{Q_Nw_n}{L^4}^2+\Sob{u}{L^\infty}\Sob{A^{1/2} Q_Nw_n}{L^2}+\Sob{Q_Nw_n}{L^4}\Sob{A^{1/2} u}{L^4}\notag\\
        &\leq CN\Sob{A^{1/2}Q_N w_n}{L^2}\Sob{Q_Nw_n}{L^2}+C\Sob{Au}{L^2}^{1/2}\Sob{u}{L^2}^{1/2}\Sob{A^{1/2} Q_Nw_n}{L^2}\notag\\
        &\quad+\Sob{A^{1/2} Q_Nw_n}{L^2}^{1/2}\Sob{Q_Nw_n}{L^2}^{1/2}\Sob{Au}{L^2}^{1/2}\Sob{A^{1/2} u}{L^2}^{1/2}\notag\\
        &\leq C\Sob{A^{1/2}Q_N w_n}{L^2}^2+C\Sob{Au}{L^2}^{1/2}\Sob{u}{L^2}^{1/2}\Sob{A^{1/2} Q_Nw_n}{L^2}+\frac{C}{N^{1/2}}\Sob{A^{1/2} Q_Nw_n}{L^2}\Sob{Au}{L^2}^{1/2}\Sob{A^{1/2} u}{L^2}^{1/2}\notag\\
        &\leq C_0\nu(\ts_1+\tG)^{1/2}\tG\left(1+\frac{\Sob{A^{1/2} Q_Nw_n}{L^2}}{\nu}\right)\Sob{A^{1/2} Q_Nw_n}{L^2},
    \end{align}
for some universal constant $C_0>0$, independent of $n$, for all $t\geq t_{n-1}$. Note that we also invoked the assumption that $u$ belongs to the absorbing ball in $H^1_\s$ and $H^2_\s$ (see \eqref{def:rad:H1}, \eqref{est:rad:H2:gen}, respectively) in obtaining the final inequality. It is at this point that one applies \eqref{est:sensitivity:rough} in order to properly close the recursive estimate. In order to rigorously carry out this argument, let us therefore prove that \eqref{est:sensitivity:rough} indeed holds.

\begin{Rmk}\label{rmk:high}
In the case when the force contains modes beyond those that are observed, one can identify an obstruction that precludes a proof in the manner described above. Suppose that $f\neq P_Nf$. Then modify the ansatz \eqref{eq:update:n} for the force at each stage by removing the projection onto low modes. In particular, re-define $f_n$ so that
    \begin{align}\notag
        f_n=\frac{d}{dt}u_n+\nu Au_n+B(u_n,u_n).
    \end{align}
Then the model error becomes
    \begin{align}\label{eq:model:error:general}
        g_n=\frac{d}{dt}Q_Nw_n+\nu AQ_Nw_n+B(Q_Nw_n,Q_Nw_n)+DB(u)Q_Nw_n.
    \end{align}
Upon applying the complementary projection, $Q_N$, to \eqref{eq:wn}, and combining the result with \eqref{eq:model:error:general}, we see that
    \begin{align}\notag
        Q_Ng_n=-B^N(P_Nw_n,P_Nw_n)-B^N(P_Nw_n,Q_Nw_n)-B^N(Q_Nw_n,P_Nw_n)-DB^N(u)P_Nw_n+Q_Ng_{n-1},
    \end{align}
where $B^N=Q_NB$ and $DB^N=Q_NDB$. Due to the presence of $Q_Ng_{n-1}$ on the right-hand side, one cannot expect to obtain a convergent recursive relation of the form $\Sob{Q_Ng_n}{L^2}\leq \be\Sob{Q_Ng_{n-1}}{L^2}$, for some $\be\in(0,1)$. Although one can establish an estimate of $Q_Ng_n$ from this relation that is of the form $\Sob{Q_Ng_n-Q_Ng_{n-1}}{L^2}\leq O_N(\Sob{A^{1/2}Q_Nw_n}{L^2})$, where the suppressed constant has a favorable dependence on $N$, a subsequent analysis will nevertheless be unable to establish a suitable recursion relation since we will only ever have access to an estimate of the form \eqref{est:sensitivity:rough}.
\end{Rmk}

\section{Sensitivity Analysis}\label{sect:sensitivity}
We establish a more precise form of the crucial estimates \eqref{est:sensitivity:rough}, which form the bridge to the desired recursion for the model error at each stage of the approximation. For this, we recall the notation introduced in \cref{sect:outline}. In particular, we prove the following.

\begin{Prop}\label{prop:sensitivity}
There exist universal constants $\overline{c}_0, \underline{c}_0\geq1$ such that if $\mu, N$ satisfy
    \begin{align}\label{cond:mu:N:sensitivity}
            \underline{c}_0\left(\ts_1+\tG\right)\tG^2\leq \frac{\mu}{\nu}\leq \overline{c}_0 N^2,
    \end{align}
then for each $n\geq1$, there exist relaxation periods, $\rho_n=t_n-t_{n-1}$,  for some $t_n>t_{n-1}$, such that
    \begin{align}\label{est:sensitivity:H1}
        \sup_{t\in I_n}\left(\frac{\Sob{A^{1/2} w_n(t)}{L^2}}{\nu}\right)
        &\leq \left(\frac{2C_1\nu}{\mu}\right)^{1/2}\left(\frac{\sup_{t\in I_{n-1}}\Sob{g_{n-1}(t)}{L^2}}{\nu^2}\right),
    \end{align}
where $I_n:=[t_{n-1}+\rho_n,\infty)$, for some universal constant $C_1\geq1$, independent of $n$. Moreover, $\rho_n$ satisfies \eqref{def:relaxation}.
\end{Prop}

\begin{proof}

Fix $n\geq1$. The enstrophy-balance for $w_n$ is obtained by taking the $L^2$--inner product of \eqref{eq:wn} by $Aw_n$, which yields
    \begin{align}\label{eq:wn:enstrophy}
        \frac{1}2\frac{d}{dt}\Sob{A^{1/2} w_n}{L^2}^2+\nu\Sob{Aw_n}{L^2}^2&=-\lb B(u,w_n),Aw_n\rb-\lb B(w_n,u),Aw_n\rb+\lb g_{n-1},Aw_n\rb-\mu \Sob{A^{1/2} P_Nw_n}{L^2}^2\notag\\
        &=I+II+III+IV.
    \end{align}
Observe that by interpolation, the Cauchy-Schwarz inequality, Poincar\'e's inequality, and \eqref{def:rad:H1}, \eqref{est:rad:H2:gen}, we may estimate
    \begin{align}
        |I|&\leq C\Sob{u}{L^\infty}\Sob{A^{1/2} w_n}{L^2}\Sob{Aw_n}{L^2}\notag\\
        &\leq C\Sob{Au}{L^2}^{1/2}\Sob{u}{L^2}^{1/2}\Sob{A^{1/2} w_n}{L^2}\Sob{Aw_n}{L^2}\notag\\
        &\leq C\mu\left(\frac{\nu}{\mu}\right)\left(\ts_1+\tG\right)\tG^2\Sob{A^{1/2} w_n}{L^2}^2+\frac{\nu}{100}\Sob{Aw_n}{L^2}^2.\notag
    \end{align}
Similarly, we obtain
    \begin{align}
        |II|
        &\leq C\Sob{A^{1/2} w_n}{L^2}^{1/2}\Sob{w_n}{L^2}^{1/2}\Sob{Au}{L^2}^{1/2}\Sob{A^{1/2}u}{L^2}^{1/2}\Sob{Aw_n}{L^2}\notag\\
        &\leq C\mu\left(\frac{\nu}{\mu}\right)\left(\ts_1+\tG\right)\tG^2\Sob{A^{1/2} w_n}{L^2}^2+\frac{\nu}{100}\Sob{Aw_n}{L^2}^2\notag.
    \end{align}
On the other hand, by the Cauchy-Schwarz inequality, we have
    \begin{align}
        |III|&\leq C\frac{\Sob{g_{n-1}}{L^2}^2}{\nu}+\frac{\nu}{100}\Sob{Aw_n}{L^2}^2\notag.
    \end{align}
Lastly, we have
    \begin{align}
        IV&=-\mu\Sob{A^{1/2} w_n}{L^2}^2+\mu\Sob{A^{1/2} Q_Nw_n}{L^2}^2\notag\\
        &\leq -\mu\Sob{A^{1/2} w_n}{L^2}^2+\frac{\mu}{N^2}\Sob{Aw_n}{L^2}^2\notag.
    \end{align}
Upon combining the estimates for $I$--$IV$ and invoking \eqref{cond:mu:N:sensitivity}, where $\underline{c}_0, \overline{c}_0\geq1$ are chosen appropriately relative to the constants $C$ appearing above, we arrive at
    \begin{align}
        \frac{d}{dt}\Sob{A^{1/2} w_n}{L^2}^2+\frac{3}2\nu\Sob{Aw_n}{L^2}^2+\frac{3}2\mu\Sob{A^{1/2} w_n}{L^2}^2&\leq C\nu^3\left(\frac{\Sob{g_{n-1}}{L^2}}{\nu^2}\right)^2.\notag
    \end{align}
An application of Gronwall's inequality yields
    \begin{align}\label{est:wn:H1}
        \left(\frac{\Sob{A^{1/2} w_n(t)}{L^2}}{\nu}\right)^2
        &\leq e^{-\mu(t-t_n)}e^{-\mu\rho_n}\left(\frac{\Sob{A^{1/2} w_n(t_{n-1})}{L^2}}{\nu}\right)^2+ C_1\left(\frac{\nu}{\mu}\right)\left(\frac{\sup_{t\geq t_{n-1}}\Sob{g_{n-1}(t)}{L^2}}{\nu^2}\right)^2,
    \end{align}
for all $t\geq t_n$, for some $C_1\geq1$ independent of $n$. Now recall that $w_n(t_{n-1})=v^0_n-u(t_{n-1})$. We choose $\rho_n>0$ such that
    \begin{align}\label{def:relaxation}
        \rho_n\geq\frac{1}{\mu}\ln\left[\left(\frac{\mu}{C_1\nu}\right)\left(\frac{\nu\Sob{A^{1/2} w_n(t_{n-1})}{L^2}}{\sup_{t\in I_{n-1}}\Sob{g_{n-1}(t)}{L^2}}\right)^2\right].
    \end{align}
Returning to \eqref{est:wn:H1}, it follows that
    \begin{align}
     \sup_{t\in I_n}\left(\frac{\Sob{A^{1/2} w_n(t)}{L^2}}{\nu}\right)
        &\leq \left(\frac{2C_1\nu}{\mu}\right)^{1/2}\left(\frac{\sup_{t\in I_{n-1}}\Sob{g_{n-1}(t)}{L^2}}{\nu^2}\right),\notag
    \end{align}
as desired.

\end{proof}

\section{Proofs of Convergence}\label{sect:proof}

Let us assume that $f=P_Nf$. At the initializing stage, $n=0$, we consider any $f_0\in C([0,\infty);L^2_\s)\cap L^\infty([0,\infty;);L^2_\s)$. Recall that in each subsequent stage, we will produce a new approximation, $f_n$, for the force via the ansatz \eqref{eq:update:n}. We proceed by induction.

\begin{proof}[Proof of \cref{thm:low}] Fix $\be\in(0,1)$.  We choose $c_0$ to satisfy
    \begin{align}\notag
        c_0\geq\max\left\{\underline{c}_0, 2C_1\left(\frac{C_0}{\be}\right)^2\left[1+\left(\frac{2C_1}{\underline{c}_0(\ts_1+\tG)}\right)^{1/2}\gam_0\right]^2\right\},
    \end{align} 
where $C_0, C_1, \underline{c}_0$ are the universal constants appearing in \eqref{est:error:low:modes}, \eqref{cond:mu:N:sensitivity}, \eqref{est:sensitivity:H1}, respectively, and $\gam_0$ denotes the initial relative error defined by \eqref{def:error:rel}. Fix any $c_1\leq \til{c}$, where $\til{c}$ is the constant appearing in \eqref{cond:mu:N:ng}. AThen assume that $\mu, N$ satisfies \eqref{cond:mu:N:low} with $c=c_0c_1^{-1}$. Then choose $\mu$ such thaty 
    \begin{align}\label{cond:mu:N:low:bound}
        c_0(\ts_1+\tG)\tG^2<\be^2\frac{\mu}{\nu}\leq c_1 \be^2 N^2,
    \end{align}
Observe that $N$ satisfies \eqref{cond:mu:N:low} with $c=c_0c_1^{-1}$. Since $c_0\geq \underline{c}_0$ by choice, it immediately follows that \eqref{cond:mu:N:sensitivity} also holds.

Let $n=1$. Denote $\mathcal{R}_N^{(0)}:=g_0=f_0-f$. Observe that from \eqref{def:Grashof}, \eqref{def:error:rel}, we have
    \begin{align}\label{est:uniform:base}
       \sup_{t\geq t_0}\Sob{\mathcal{R}_N^{(0)}(t)}{L^2}=\gam_0\nu^2\tG.
    \end{align}
Combining \cref{prop:sensitivity}, \eqref{est:error:low:modes} for $n=1$, and \eqref{est:uniform:base} ensures that there exists a relaxation period $\rho_1>0$, for some $t_1\geq t_0$, 
such that
    \begin{align}\label{eq:base:case:low}
    \sup_{t\geq I_1}\Sob{\mathcal{R}_N^{(1)}(t)}{L^2}\leq (2C_1)^{1/2}C_0(\ts_1+\tG)^{1/2}\tG\left(1+\left(\frac{2C_1\nu}{\mu}\right)^{1/2}\gam_0\tG\right)\left(\frac{\nu}{\mu}\right)^{1/2}\sup_{t\geq I_0}\Sob{\mathcal{R}_N^{(0)}(t)}{L^2},
    \end{align}
where $I_1=[t_0+\rho_1,\infty)$. 

Further assume that $\mu$ satisifes
    \begin{align}\label{cond:mu:low}
        \be^2\frac{\mu}{\nu}\geq 2C_1C_0^2\left((\ts_1+\tG)^{1/2}+\left(\frac{2C_1}{\underline{c}_0}\right)^{1/2}\gam_0\right)^2\tG,
    \end{align}
where $C_1$ is the same constant appearing in \eqref{eq:base:case:low}. Then \eqref{eq:base:case:low}, \eqref{cond:mu:N:sensitivity}, and \eqref{cond:mu:low} imply
    \begin{align}\label{est:base:case:low}
        \sup_{t\geq I_1}\Sob{\mathcal{R}_N^{(1)}(t)}{L^2}\leq \be\left(\sup_{t\geq I_0}\Sob{\mathcal{R}_N^{(0)}(t)}{L^2}\right).
    \end{align}
This establishes the base case.

Now suppose that for all $k=1,\dots, n$, there exist relaxation periods $\rho_k$, such that
    \begin{align}\label{eq:induct:hyp:low}
    \sup_{t\geq I_k}\Sob{\mathcal{R}_N^{(k)}(t)}{L^2}\leq \be\sup_{t\geq I_{k-1}}\Sob{\mathcal{R}_N^{(k-1)}(t)}{L^2}.
    \end{align}
With \eqref{est:uniform:base}, it follows that 
    \begin{align}\label{est:bdd:seq}
        \sup_{t\geq I_k}\Sob{\mathcal{R}_N^{(k)}(t)}{L^2}\leq \be^k\gam_0\tG\leq \gam_0\tG,
    \end{align}
for all $k=1,\dots, n$. We may thus deduce from \cref{prop:sensitivity}, \eqref{est:error:low:modes}, and \eqref{est:bdd:seq} that
    \begin{align}\label{est:converge:lowmodes}
        \sup_{t\geq I_{n+1}}\Sob{\mathcal{R}_N^{(n+1)}(t)}{L^2}\leq (2C_1)^{1/2}C_0(\ts_1+\tG)^{1/2}\tG\left(1+\left(\frac{2C_1\nu}{\mu}\right)^{1/2}\gam_0\tG\right)\left(\frac{\nu}{\mu}\right)^{1/2}\sup_{t\geq I_n}\Sob{\mathcal{R}_N^{(n)}(t)}{L^2}.
    \end{align}
An application of \eqref{cond:mu:low} and \eqref{cond:mu:N:sensitivity} then implies
    \begin{align}\notag
     \sup_{t\geq I_{n+1}}\Sob{\mathcal{R}_N^{(n+1)}(t)}{L^2}\leq \be\left(\sup_{t\geq I_n}\Sob{\mathcal{R}_N^{(n)}(t)}{L^2}\right),
    \end{align}
as desired. This completes the induction.
\end{proof}

\begin{Rmk}\label{rmk:relaxation}
From the proof of \cref{thm:low}, we can obtain more informative estimates on the relaxation periods, $\rho_n$. Indeed, by \eqref{def:relaxation} and \eqref{est:uniform:base}, we see that the first relaxation period satisfies
    \begin{align}\label{est:relaxation:base}
        \rho_1\geq \frac{1}{\mu}\ln\left[\left(\frac{\mu}{C_1\nu}\right)\frac{1}{\gam_0^2\tG^2}\left(\frac{\Sob{A^{1/2}(v^0_1-u(t_0))}{L^2}}{\nu}\right)^2\right].
    \end{align}
In subsequent stages, $n\geq1$, by \eqref{def:relaxation} and \eqref{est:sensitivity:H1} applied at the preceding stage, we may deduce that the relaxation periods satisfy
    \begin{align}\label{est:relaxation:n}
        \rho_n\geq \frac{1}{\mu}\ln\left(\frac{\sup_{t\in I_{n-2}}\Sob{\mathcal{R}_N^{(n-2)}(t)}{L^2}}{\sup_{t\in I_{n-1}}\Sob{\mathcal{R}_N^{(n-1)}(t)}{L^2}}\right)\geq -\frac{1}{\mu}\ln\be.
    \end{align}
One may then observe that \eqref{est:relaxation:base} and \eqref{est:relaxation:n} follow a consistent pattern if we allow ourselves to make special, but nevertheless natural choices for initial force ansatz, $f_0$, and the initial data, $v_1^0$, of the first nudged system. Indeed, suppose that the solution, $u$, of \eqref{eq:nse} possesses additional regularity, for instance, $Au(t_0)\in L^2_\s$. Now initialize the nudged system \eqref{eq:step:1} with $v_1^0=P_Nu(t_0)$ and let the initial guess $f_0$ be given by \eqref{def:force:initial}. Then the first relaxation period satisfies
    \begin{align}\label{est:relaxation:refined}
        \rho_1\geq \frac{1}{\mu}\ln\left[\left(\frac{c_1}{C_1}\right)\left(\frac{\nu^2}{\sup_{t\in I_0}\Sob{\mathcal{R}_N^{(0)}(t)}{L^2}}\right)^2\right],
    \end{align}
where we applied the Poincar\'e inequality, the upper bound in \eqref{cond:mu:N:low:bound}, and \eqref{def:Grashof}.
\end{Rmk}

Now we prove a variation on the time-independent case, which allows one to recycle the existing data.  

\begin{proof}[Proof sketch of \cref{thm:recycle}]
We modify the algorithm in \cref{sect:algorithm} as follows: Let $J_0=I_0=[t_0,\infty)$. For convenience, we assume that $t_0=0$. Suppose that $\{P_Nu(t)\}_{t\in J_0}$ is known. At stage $n=1$, we suppose initial data $v_1^0$ is given and solve \eqref{eq:step:1} to produce the solution $v_1$ over $J_0$. We then define \eqref{def:vel:1} over $J_0$ and immediately generate $f_1(t)$ via \eqref{eq:update:1} over $J_0$. To define the stage $1$ approximation to $f$, we evaluate $f_1(t)$ after a transient period of length $\rho_1>0$ to obtain $f_1:=f_1(\rho_1)$.

In subsequent stages $n\geq1$, we generate $v_n$ over $J_0$ via \eqref{eq:step:n}, where $f_{n-1}:=f_{n-1}(\rho_{n-1})$ was generated from the preceding stage, $n-1$. We then define the stage-$n$ approximation to the true state over $J_0$ by $u_n=P_Nu+Q_Nv_n$.  We generate a new force, $f_n(t)$, via the ansatz \eqref{eq:update:n}. The stage-$n$ approximation of $f$ is then given by $f_n(\rho_n)$, for some $\rho_n>0$.

To assess the error, we once again form the synchronization error, $w_n=v_n-u$, and the model error, $g_n(t)=f_n(t)-f$. For each $n\geq1$, the identity \eqref{eq:error:low:modes} still holds. Consequently, \eqref{est:error:low:modes} holds as well. The remaining ingredient to establish convergence is a time-independent force analog of \cref{prop:sensitivity}.

Supposing that \eqref{cond:mu:N:sensitivity} holds with $\tG\mapsto G$, the proof of \cref{prop:sensitivity} proceeds the same way, except that the analysis is carried out entirely over the interval $J_0$, to arrive at the analog of \eqref{est:wn:H1}:
    \begin{align}\label{est:wn:H1:time:independent}
         \left(\frac{\Sob{A^{1/2} w_n(\rho_n)}{L^2}}{\nu}\right)^2
        &\leq e^{-\mu\rho_n}\left(\frac{\Sob{A^{1/2} w_n(0)}{L^2}}{\nu}\right)^2+ C_1\left(\frac{\nu}{\mu}\right)\left(\frac{\Sob{g_{n-1}}{L^2}}{\nu^2}\right)^2.
    \end{align}
One chooses $\rho_n>0$ according to
    \begin{align}\notag
        \rho_n\geq \frac{1}{\mu}\left[\frac{\mu}{C_1\nu}\left(\frac{\nu\Sob{A^{1/2}w_n(0)}{L^2}}{\Sob{g_{n-1}}{L^2}}\right)^2\right],
    \end{align}
so that \eqref{est:wn:H1:time:independent} yields
    \begin{align}
      \left(\frac{\Sob{A^{1/2} w_n(\rho_n)}{L^2}}{\nu}\right)^2
        &\leq 2C_1\left(\frac{\nu}{\mu}\right)\left(\frac{\Sob{g_{n-1}}{L^2}}{\nu^2}\right)^2.\notag
    \end{align}
It is at this time $t=\rho_n$, that we evaluate \eqref{eq:update:n}. 

After these adjustments, it is now clear that the proof of \cref{thm:recycle} follows in analogous manner to the the proof \cref{thm:low}, mutatis mutandis.
\end{proof}

\subsection*{Acknowledgments} The author thanks the reviewer for their helpful comments. The author also graciously acknowledges the generosity and support of the ADAPT group. Support for this project was also provided by the National Science Foundation through NSF-DMS 2213363 and NSF-DMS 2206491, as well as PSC-CUNY Award 65187-00 53, which is jointly funded by The Professional Staff Congress and The City University of New York.

\appendix

\section{Uniform-in-time Estimates for Palenstrophy}\label{sect:app:apriori}

We now provide uniform-in-time estimates in $H^2$ for the reference flow field when the external force field is time-dependent. We begin by establishing an alternative form of the standard enstrophy balance of \eqref{eq:nse} as it is presented in \cref{thm:nse:wellposed}. Indeed, upon taking the $L^2$-inner product of \eqref{eq:nse} with $Au$, we obtain
    \begin{align}\notag
        \frac{1}2\frac{d}{dt}\Sob{A^{1/2}u}{L^2}^2+\nu\Sob{Au}{L^2}^2=\lb f,Au\rb.
    \end{align}
By the Cauchy-Schwarz inequality and \eqref{def:Grashof}, we see that
    \begin{align}\notag
        |\lb f,Au\rb|\leq \nu^3\tG^2+\frac{\nu}{4}\Sob{Au}{L^2}^2.
    \end{align}
By the Poincar\'e inequality, it follows that
    \begin{align}\notag
        \frac{d}{dt}\left(e^{\nu t}\Sob{A^{1/2}u}{L^2}^2\right)+\frac{\nu}2\Sob{Au}{L^2}^2\leq 2\nu^3\tG^2e^{\nu t}.
    \end{align}
Integrating over $t\geq t_0$ yields
    \begin{align}\notag
        \Sob{A^{1/2}u(t)}{L^2}^2+\frac{\nu}2\int_{t_0}^te^{-\nu(t-s)}\Sob{Au(s)}{L^2}^2\leq \Sob{A^{1/2}u(t_0)}{L^2}^2e^{-\nu(t-t_0)}+2\nu^2\tG^2(1-e^{-\nu(t-t_0)}),
    \end{align}
which is \eqref{est:nse:H1}, as desired. Thus, for $u_0\in\til{B}_1$, we may deduce
    \begin{align}\label{est:enstrophy:bound}
        \Sob{A^{1/2}u(t)}{L^2}^2+\frac{\nu}2\int_{t_0}^te^{-\nu(t-s)}\Sob{Au(s)}{L^2}^2\leq 2\nu^2\tG^2,
    \end{align}
for all $t\geq t_0\geq0$.

\begin{proof}[Proof of \eqref{def:rad:H2:gen}]
Upon taking the $L^2$-inner product of \eqref{eq:nse} with $A^2u$, we obtain
    \begin{align}
        \frac{1}2\frac{d}{dt}\Sob{Au}{L^2}^2+\nu\Sob{A^{3/2}u}{L^2}^2&=-\lb B(u,u),A^2u\rb+\lb f,A^2u\rb=I+II.\notag
    \end{align}
Observe that integration by parts multiple times yields
    \begin{align}
        I
        &=-\lb B(Au,u),Au\rb-2\sum_{\ell=1,2}\lb B(\bdy_\ell u,\bdy_\ell u),Au\rb\notag,
    \end{align}
where we applied the identity $\lb B(u,v),v\rb=0$. It then follows from H\"older's inequality, interpolation, Young's inequality, and \eqref{est:nse:H1} that
    \begin{align}
        |I|&\leq 3\Sob{A^{1/2}u}{L^2}\Sob{Au}{L^4}^2\leq 3c_L\Sob{A^{1/2} u}{L^2}\Sob{Au}{L^2}\Sob{A^{3/2}u}{L^2}\leq \frac{\nu}{8}\Sob{A^{3/2}u}{L^2}^2+36c_L^2\nu\tG^2\Sob{Au}{L^2}^2,\notag
    \end{align}
where $c_L\geq1$ is the associated constant of interpolation, i.e., $\Sob{Au}{L^4}^2\leq c_L\Sob{Au}{L^2}\Sob{A^{3/2}u}{L^2}$ On the other hand, we treat $II$ with integration by parts and the Cauchy-Schwarz inequality to obtain
    \begin{align}
        |II|&\leq|\lb A^{1/2} f,A^{3/2} u\rb|\leq 2\nu^3\left(\frac{\Sob{A^{1/2} f}{L^\infty_tL^2_x}}{\nu^2}\right)^2+\frac{\nu}{8}\Sob{A^{3/2}u}{L^2}^2\leq 2\nu^3\ts_1^2\tG^2+\frac{\nu}{8}\Sob{A^{3/2}u}{L^2}^2.\notag
    \end{align}
We now combine the estimates $I,II$ and apply Poincar\'e's inequality to arrive at
    \begin{align}
        \frac{d}{dt}\Sob{Au}{L^2}^2+\nu\Sob{Au}{L^2}^2+\frac{\nu}2\Sob{A^{3/2}u}{L^2}^2\leq 36c_L^2\nu\tG^2\Sob{Au}{L^2}^2+2\nu^3\ts_1^2\tG^2.\notag
    \end{align}
An application of Gronwall's inequality yields
    \begin{align}\notag
        \Sob{Au(t)}{L^2}^2+&\frac{\nu}2\int_{t_0}^te^{-\nu(t-s)}\Sob{A^{3/2}u(s)}{L^2}^2ds
        \notag\\
        &\leq \Sob{Au(t_0)}{L^2}^2e^{-\nu(t-t_0)}+36c_L^2\nu\tG^2\int_{t_0}^te^{-\nu(t-s)}\Sob{Au(s)}{L^2}^2ds+2\nu^2\ts_1^2\tG^2(1-e^{-\nu(t-t_0)}),\notag
    \end{align}
for all $t \geq t_0\geq0$. We now apply \eqref{est:enstrophy:bound} to bound
    \begin{align}\notag
        \Sob{Au(t)}{L^2}^2+\frac{\nu}2\int_{t_0}^te^{-\nu(t-s)}\Sob{A^{3/2}u(s)}{L^2}^2ds
        \leq \Sob{Au(t_0)}{L^2}^2e^{-\nu(t-t_0)}+\til{c}_2^2\nu^2(\ts_1+\tG)^2\tG^2,
    \end{align}
where $\til{c}_2=12c_L$, which holds for all $t\geq t_0\geq0$. 

Therefore, if $\Sob{Au_0}{L^2}\leq \al\til{c}_2\nu(\ts_1+\tG)\tG$, for any $\al>0$, then
    \begin{align}\notag
        \Sob{Au(t)}{L^2}^2\leq \til{c}_2^2(1+\al^2)\nu^2(\ts_1+\tG)^2\tG^2,
    \end{align}
for all $t\geq0$. This completes the proof.
\end{proof}

\newcommand{\etalchar}[1]{$^{#1}$}
\providecommand{\bysame}{\leavevmode\hbox to3em{\hrulefill}\thinspace}
\providecommand{\MR}{\relax\ifhmode\unskip\space\fi MR }
\providecommand{\MRhref}[2]{%
  \href{http://www.ams.org/mathscinet-getitem?mr=#1}{#2}
}

\vfill 

\noindent Vincent R. Martinez\\
{\footnotesize
Department of Mathematics \& Statistics\\
CUNY Hunter College \\
Department of Mathematics \\
CUNY Graduate Center \\
Web: \url{http://math.hunter.cuny.edu/vmartine/}\\
Email: \url{vrmartinez@hunter.cuny.edu}\\
}

\end{document}